\documentclass[a4paper,reqno]{amsart}
\usepackage{amssymb, amsmath, amscd}
\usepackage[final]{graphicx}
\usepackage{float}\usepackage{wrapfig}
\usepackage{bbm}
\def\pone{\mathbbm{1}}
\def\MM{{\mathcal{M}}}
\def\PP{{\Theta}}
\newtheorem{theorem}{Theorem}[section]

\newtheorem{lemma}[theorem]{Lemma}

\newtheorem{definition}[theorem]{Definition}

\makeatletter
 \@addtoreset{equation}{section}
\makeatother
\begin{document}
\title[Completeness]
{Affine Killing complete and geodesically complete homogeneous affine surfaces}
\author
{P. B. Gilkey, J. H. Park, and X. Valle-Regueiro}
\address{PBG: Mathematics Department, University of Oregon, Eugene OR 97403-1222, USA}
\email{gilkey@uoregon.edu}
\address{JHP: Department of
Mathematics, Sungkyunkwan University, Suwon, 16419 Korea.}
\email{parkj@skku.edu}
\address{XVR: Faculty of Mathematics,
University of Santiago de Compostela,
15782 Santiago de Compostela, Spain}
\email{javier.valle@usc.es}
\thanks{PBG and XVR were supported by Project MTM2016-75897-P (AEI/FEDER, UE),
JHP was supported by  Basic Science Research Program through
the National Research Foundation of Korea (NRF)
 funded by the Ministry of Education (NRF-2016R1D1A1B03930449).}
\subjclass[2010]{53C21}
\keywords{quasi-Einstein equation, geodesic completeness, Killing completeness}
\begin{abstract} An affine manifold is said to be geodesically complete if all affine geodesics
extend for all time. It is said to be affine Killing complete if the integral curves for
any affine Killing vector field extend for all time. We use the solution space of the quasi-Einstein
equation to examine these concepts in the setting of homogeneous affine surfaces.\end{abstract}
\maketitle

\section{Introduction}
Let $M$ be a connected smooth manifold of dimension $m$ which is equipped with a torsion free connection $\nabla$ on the tangent
bundle of $M$; the pair $\mathcal{M}=(M,\nabla)$ is said to be an {\it affine manifold}.
If $g$ is a pseudo-Riemannian metric on $M$, then the corresponding affine structure is obtained
by taking $\nabla$ to be the Levi-Civita connection. However, not all affine structures arise
in this fashion; such structures are said to be {\it not metrizable}. A diffeomorphism
from one affine manifold to another is said to be an {\it affine map} if
it intertwines the two connections.

Let $\Phi_t^X$ be the local 1-parameter flow of a vector field $X$ on $M$.  The following 3 conditions are equivalent and
if any is satisfied, then $X$ is said to be
an {\it affine Killing vector field} (see Kobayashi and Nomizu~\cite{KN63}):
\begin{enumerate}
\item $\Phi_t^X$ is an affine map where defined.
\item The Lie derivative of $\nabla$ with respect to $X$ vanishes.
\item $[X,\nabla_YZ]-\nabla_Y[X,Z]-\nabla_{[X,Y]}Z=0$ for all smooth vector fields $Y$ and $Z$.
\end{enumerate}
Let $\mathfrak{K}(\mathcal{M})$ be the set of affine Killing vector fields. The
Lie bracket gives $\mathfrak{K}(\mathcal{M})$ the structure of a real Lie algebra. Furthermore,
if $X\in\mathfrak{K}(\mathcal{M})$, if $X(P)=0$, and if $\nabla X(P)=0$, then $X\equiv0$. Consequently,
$\dim\{\mathfrak{K}(\mathcal{M})\}\le m+m^2$; if equality holds, then $\mathcal{M}$ is flat.
An affine Killing vector field is said to be {\it complete} if the flow $\Phi_t^X$ exists for all $t$.

Let $\operatorname{Aff}(\mathcal{M})$ be the Lie group of all affine diffeomorphisms of $\mathcal{M}$.
The Lie algebra of $\operatorname{Aff}(\mathcal{M})$ is the space of complete affine Killing vector fields.
We say that $\mathcal{M}$ is {\it affine Killing complete} if all affine Killing vector fields are complete or, equivalently,
the Lie algebra of $\operatorname{Aff}(\mathcal{M})$ is $\mathfrak{K}(\mathcal{M})$.
Consequently, determining whether or not $\mathcal{M}$ is affine Killing complete is a central geometrical question.

A smooth curve $\sigma(t)$ in $M$ is said to be a {\it geodesic} if $\nabla_{\dot\sigma}\dot\sigma=0$.
We adopt the {\it Einstein convention} and sum over repeated indices to expand
$\nabla_{\partial_{x^i}}\partial_{x^j}=\Gamma_{ij}{}^k\partial_{x^k}$ in a system of local
coordinates. If $\sigma(t)=(x^1(t),\dots,x^m(t))$, then $\sigma$ is a geodesic if and only if the {\it geodesic equation} is satisfied, i.e.
\begin{equation}\label{E1.a}
\ddot\sigma^k+\Gamma_{ij}{}^k\dot\sigma^i\dot\sigma^j=0\text{ for all }k\,.
\end{equation}
$\mathcal{M}$ is said to be {\it geodesically complete} if every geodesic extends for infinite time.
Any geodesically complete affine manifold is affine Killing complete  (see Kobayashi and Nomizu~\cite{KN63}) but the converse fails as we shall
see presently.

If $\operatorname{Aff}(\mathcal{M})$ acts transitively on $M$, then $\mathcal{M}$ is said to be
{\it affine homogeneous}; there is a corresponding local theory if the diffeomorphisms in question are only assumed
to be locally defined. The classification of locally homogeneous affine surfaces by Opozda~\cite{Op04}
may be described as follows. Up to isomorphism, there
are two simply connected Lie groups of dimension $2$,
the translation group $\mathbb{R}^2$ and the $ax+b$ group $\mathbb{R}^+\times\mathbb{R}$.
A left invariant affine structure on $\mathbb{R}^2$ (resp. on $\mathbb{R}^+\times\mathbb{R}$) is said to be {\it Type~$\mathcal{A}$}
(resp. {\it Type~$\mathcal{B}$}). These geometries are globally homogeneous; $\operatorname{Aff}(\cdot)$ acts transitively
on such geometries. Every locally
homogeneous affine surface is either modeled on a Type~$\mathcal{A}$ geometry, on a Type~$\mathcal{B}$
geometry, or on the geometry of the round sphere $S^2$ in $\mathbb{R}^3$ with the Levi-Civita connection.

Any Riemannian metric on a compact manifold is complete. Thus
the sphere is geodesically complete. Similarly, any vector field on a compact manifold is complete and thus
the sphere is Killing complete.
For that reason, we will concentrate on studying the Type~$\mathcal{A}$ and Type~$\mathcal{B}$ geometries in this paper.
We emphasize that geodesic completeness (resp., affine Killing completeness) is equivalent to prolonging a system of second order
(resp., first order) non-linear ODEs. Even in the homogeneous setting these equations can be quite unmanageable. Consequently, instead
of a direct approach, we shall follow a different ansatz making use of the affine quasi-Einstein equation.
We will examine Killing completeness for both the Type~$\mathcal{A}$ and the Type~$\mathcal{B}$ geometries.
However, we will examine geodesic completeness only in the context of the Type~$\mathcal{A}$ geometries
as the quasi-Einstein equation proves not to be terribly useful in studying geodesic completeness for the Type~$\mathcal{B}$
geometries.

\subsection{The Hessian, the curvature, and the quasi-Einstein equation} Set
$$
\mathcal{H}\phi=(\partial_{x^i}\partial_{x^j}\phi-\Gamma_{ij}{}^k\partial_{x^k}\phi)dx^i\otimes dx^j\in S^2(M)\,.
$$
Define the curvature operator $R(\cdot,\cdot)$ and the Ricci tensor $\rho(\cdot,\cdot)$ by setting:
$$
R(u,v):=\nabla_u\nabla_v-\nabla_v\nabla_u-\nabla_{[u,v]} \text{ and }
\rho(x,y):=\operatorname{Tr}\{z\rightarrow R(z,x)y\}\,.
$$
As the Ricci tensor need not be symmetric,
we introduce the symmetrization $\rho_s(x,y):=\frac12\{\rho(x,y)+\rho(y,x)\}$. Let
$\mathcal{Q}(\mathcal{M})$ be the solution space of the {\it quasi-Einstein equation}:
$$
\mathcal{Q}(\mathcal{M}):=\{\phi\in C^\infty(M):\mathcal{H}\phi+\phi\rho_s=0\}\,.
$$

\subsection{Strong projective equivalence}
We say that two affine connections $\nabla$ and $\tilde\nabla$ are {\it strongly projectively equivalent}
if there exists a smooth function $\varphi$
so that $\tilde\nabla_XY=\nabla_XY+Y(\varphi)X+X(\varphi)Y$. In this setting,
we shall say that $\varphi$ provides a {\it strong projective equivalence} from $\mathcal{M}=(M,\nabla)$ to
${}^\varphi\mathcal{M}:=(M,\tilde\nabla)$. We say that $\mathcal{M}$ is {\it strongly
projectively flat} if $\mathcal{M}$ is strongly projectively equivalent to a flat connection.

We will prove the following result in Section~\ref{S2}.

\begin{lemma}\label{L1.1}Let $\mathcal{M}=(\mathbb{R}^2,\nabla)$ be a Type~$\mathcal{A}$ geometry.
There exists a linear function $\varphi(x^1,x^2)=a_1x^1+a_2x^2$ which provides
a strong projective equivalence from $\mathcal{M}$ to a flat Type~$\mathcal{A}$ geometry
and which satisfies $e^{-\varphi}\in\mathcal{Q}(\mathcal{M})$.\end{lemma}

There is a close relationship between strong projective equivalence and the solutions of the quasi-Einstein equation.
We refer to Brozos-V\'{a}zquez et al.~\cite{BGGV16} and to Gilkey and Valle-Regueiro~\cite{GV18} for the proof of the
following result.
\begin{theorem}\label{T1.2}
Let $\mathcal{M}=(M,\nabla)$ be a simply connected affine surface.
\begin{enumerate}
\item If $\phi\in\mathcal{Q}(\mathcal{M})$ with $\phi(0)=0$ and $d\phi(0)=0$, then $\phi\equiv0$.
\item $\dim\{\mathcal{Q}(\mathcal{M})\}\le 3$.
\item $\dim\{\mathcal{Q}(\mathcal{M})\}=3$ if and only if $\mathcal{M}$ is strongly projectively flat.
\item $\mathcal{Q}({}^{\varphi}\mathcal{M})=e^{\varphi}\mathcal{Q}(\mathcal{M})$.
\item Let $(M,\nabla)$ and $(\tilde M,\tilde\nabla)$ be two strongly projectively flat affine surfaces and
let $\Phi$ be a diffeomorphism
from $M$ to $\tilde M$. If $\Phi^*\mathcal{Q}(\tilde M,\tilde\nabla)=\mathcal{Q}(M,\nabla)$,
then $\Phi^*\tilde\nabla=\nabla$.
\end{enumerate}\end{theorem}

By Theorem~\ref{T1.2}, $\mathcal{Q}$ transforms conformally under strong projective deformations. Since the unparameterized geodesic
structure is not altered by projective deformations, $\mathcal{Q}$ is intimately related with the affine geodesic structure in this instance.

\subsection{Type~$\mathcal{A}$ geometries} The Christoffel symbols of a Type $\mathcal{A}$ structure on $\mathbb{R}^2$
take the form $\Gamma=\Gamma(a,b,c,d,e,f)$ for $(a,b,c,d,e,f)\in\mathbb{R}^6$ where
$$\Gamma(a,b,c,d,e,f):=\left\{\begin{array}{lll}\Gamma_{11}{}^1=a,&\Gamma_{11}{}^2=b,&
\Gamma_{12}{}^1=\Gamma_{21}{}^1=c\\
\Gamma_{12}{}^2=\Gamma_{21}{}^2=d,&\Gamma_{22}{}^1=e,&\Gamma_{22}{}^2=f
\end{array}\right\}.$$
Let $\mathcal{M}(a,b,c,d,e,f)$ be the corresponding affine structure on $\mathbb{R}^2$.
\begin{definition}\label{D1.3}
\rm We define the following Type~$\mathcal{A}$
affine structures $\MM_i^j(\cdot)$ on $\mathbb{R}^2$; a direct computation then
establishes $\mathcal{Q}$:
\medbreak
$\MM_0^6:=\mathcal{M}(0,0,0,0,0,0)$, \qquad$\mathcal{Q}(\MM_0^6)=\operatorname{Span}\{\pone,x^1,x^2\}$.
\par$\MM_1^6:=\mathcal{M}(1,0,0,1,0,0)$, \qquad$\mathcal{Q}(\MM_1^6)=\operatorname{Span}\{\pone,e^{x^1},x^2e^{x^1}\}$.
\par$\MM_2^6:=\mathcal{M}(-1,0,0,0,0,1)$, \phantom{.....}$\mathcal{Q}(\MM_2^6)=\operatorname{Span}\{\pone,e^{x^2},e^{-x^1}\}$.
\par$\MM_3^6:=\mathcal{M}(0,0,0,0,0,1)$, \phantom{........}$\mathcal{Q}(\MM_3^6)=\operatorname{Span}\{\pone,x^1,e^{x^2}\}$.
\par$\MM_4^6:=\mathcal{M}(0,0,0,0,1,0)$,\phantom{........} $\mathcal{Q}(\MM_4^6)=\operatorname{Span}\{\pone,x^2,(x^2)^2+2x^1\}$.
\par$\MM_5^6:=\mathcal{M}(1,0,0,1,-1,0)$,\phantom{.....}
$\mathcal{Q}(\MM_5^6)=\operatorname{Span}\{\pone,e^{x^1}\cos(x^2),e^{x^1}\sin(x^2)\}$.
\par$\MM_1^4:=\mathcal{M}(-1,0,1,0,0,2)$,\phantom{.....}
$\mathcal{Q}(\MM_1^4)=\operatorname{Span}\{e^{x^2},x^2e^{x^2},e^{-x^1+x^2}\}$.
\par$\MM_2^4(c):=\mathcal{M}(-1,0,c,0,0,1+2c)\text{ for }c\notin\{0,-1\}$,
\par\qquad$\mathcal{Q}(\MM_2^4(c))=\operatorname{Span}\{e^{cx^2},e^{(1+c)x^2},e^{cx^2-x^1}\}$.
\par$\MM_3^4(c):=\mathcal{M}(0,0,c,0,0,1+2c)\text{ for }c\notin\{0,-1\}$, \par\qquad
$\mathcal{Q}(\MM_3^4(c))=\operatorname{Span}\{e^{cx^2},e^{(1+c)x^2},x^1e^{cx^2}\}$.
\par$\MM_4^4(c):=\mathcal{M}(0,0,1,0,c,2)$, $\mathcal{Q}(\MM_4^4(c))=\operatorname{Span}\{e^{x^2},x^2e^{x^2},
(\textstyle\frac12c(x^2)^2+x^1)e^{x^2}\}$.
\par$\MM_5^4(c):=\mathcal{M}(1,0,0,0,1+c^2,2c)$,
\par\qquad$\mathcal{Q}(\MM_5^4(c))=\operatorname{Span}\{e^{cx^2}\cos(x^2),e^{cx^2}\sin(x^2),
e^{x^1}\}$.
\par$\MM_1^2(a_1,a_2):=\mathcal{M}\left(\frac
{a_1^2+a_2-1,a_1^2-a_1,a_1a_2,a_1a_2,a_2^2-a_2,a_1+a_2^2-1}{a_1+a_2-1}\right)$ for $a_1a_2\ne0$
\par\qquad and $a_1+a_2\ne1$,
$\mathcal{Q}(\MM_1^2(a_1,a_2))=\operatorname{Span}\{e^{x^1},e^{x^2},e^{a_1x^1+a_2x^2}\}$.
\par $\MM_2^2(b_1,b_2):=\mathcal{M}\left(1+b_1,0,b_2,1,\frac{1+b_2^2}{b_1-1},0\right)$ for $b_1\ne1$,
\par\qquad
$\mathcal{Q}(\MM_2^2(b_1,b_2))=\operatorname{Span}\{e^{x^1}\cos(x^2),e^{x^1}\sin(x^2),e^{b_1x^1+b_2x^2}\}$.
\par$\MM_3^2(c):=\mathcal{M}(2,0,0,1,c,1)$ for $c\ne0$,\par\qquad\qquad
$\mathcal{Q}(\MM_3^2(c))=\operatorname{Span}\{e^{x^1},(x^1-cx^2)e^{x^1},e^{x^1+x^2}\}$.
\par$\MM_4^2(\pm1):=\mathcal{M}(2,0,0,1,\pm1,0)$,\par\qquad\qquad
$\mathcal{Q}(\MM_4^2(\pm1))=\operatorname{Span}\{e^{x^1},x^2e^{x^1},(2x^1\pm (x^2)^2)e^{x^1}\}$.
\end{definition}

\subsection{Linear equivalence and parametrization}
We say that two Type~$\mathcal{A}$ geometries $(\mathbb{R}^2,\nabla_1)$ and $(\mathbb{R}^2,\nabla_2)$ are
{\it linearly equivalent} if some element of $\operatorname{GL}(2,\mathbb{R})$ intertwines these two geometries.
The parametrization of the
Type~$\mathcal{A}$ geometries given below in Theorem~\ref{T1.4} was established by Gilkey and Valle-Regueiro \cite{GV18};
we also refer to a slightly different
parametrization given in Brozos-V\'{a}zquez, Garc\'{i}a-R\'{i}o, and  Gilkey~\cite{BGG16}.

\begin{theorem}\label{T1.4}
Let $\mathcal{M}$ be a Type~$\mathcal{A}$ structure.
\begin{enumerate}
\item The Ricci tensor of $\mathcal{M}$ is symmetric.
\item The following assertions are equivalent.
\begin{enumerate}
\item $\operatorname{Rank}\{\rho\}=2$.
\item $\dim\{\mathfrak{K}(\mathcal{M})\}=2$.
\item $\mathcal{M}$ is linearly equivalent to $\MM_i^2(\cdot)$ for some $i$.
\end{enumerate}
\item The following assertions are equivalent.
\begin{enumerate}
\item $\operatorname{Rank}\{\rho\}=1$.
\item $\dim\{\mathfrak{K}(\mathcal{M})\}=4$.
\item $\mathcal{M}$ is linearly equivalent to $\MM_i^4(\cdot)$ for some $i$.
\end{enumerate}
\item The following assertions are equivalent.
\begin{enumerate}
\item $\operatorname{Rank}\{\rho\}=0$.
\item $\dim\{\mathfrak{K}(\mathcal{M})\}=6$.
\item $\mathcal{M}$ is linearly equivalent to $\MM_i^6(\cdot)$ for some $i$.
\item $\mathcal{M}$ is flat.
\end{enumerate}\end{enumerate}
\end{theorem}

Although $\MM_i^j(\cdot)$ is not linearly equivalent to
$\MM_u^v(\cdot)$ for $(u,v)\ne(i,j)$, it can happen that
$\MM_i^j(\cdot)$ is linearly equivalent to $\MM_i^j(\cdot)$ for different values of the parameters involved;
for example, we may interchange the coordinates $x^1\leftrightarrow x^2$ to see that
$\MM_1^2(a_1,a_2)$ is linearly equivalent to $\MM_1^2(a_2,a_1)$. Giving
a precise description of the identifications describing the relevant moduli spaces is somewhat difficult and we refer for
\cite{BGG16,GV18} for further details as it will play no role here.
The notation is chosen so that $\dim\{\mathfrak{K}(\MM_i^j(\cdot))\}=j$.

\subsection{Affine Killing completeness}
We will prove the following result in Section~\ref{S4}.

\begin{theorem}\label{T1.5}
Let $\mathcal{M}=(\mathbb{R}^2,\nabla)$ be a Type~$\mathcal{A}$ surface.
Then $\mathcal{M}$ is affine Killing complete if and only if $\mathcal{M}$
is linearly equivalent to $\MM_0^6$, $\MM_4^6$, $\MM_3^4(c)$, $\MM_4^4(c)$,
or to $\MM_i^2(\cdot)$ for some $i$.
\end{theorem}

The structures $\mathcal{M}_1^6$, $\MM_2^6$, $\MM_3^6$, $\MM_5^6$,
$\MM_1^4$, $\MM_2^4(c)$, and $\MM_5^4(c)$ are, up to linear equivalence, the only affine Killing incomplete
Type~$\mathcal{A}$ structures on $\mathbb{R}^2$.
In Section~\ref{S3}, we will exhibit affine immersions of
these structures into affine Killing complete Type~$\mathcal{A}$ surfaces and show thereby these structures can be affine Killing completed.

\subsection{The geodesic equations}
In Section~\ref{S5}, we will establish the following result that reduces
 the system of geodesic equations to a single ODE in the context of Type~$\mathcal{A}$ structures on $\mathbb{R}^2$. This will simplifiy our subsequent analysis enormously; it is exactly this step which fails for the Type~$\mathcal{B}$ geometries and which renders the analysis
 of the geodesic structure of the Type~$\mathcal{B}$ geometries so difficult.

\begin{theorem}\label{T1.6} Let $\mathcal{M}$ be a Type~$\mathcal{A}$ surface.
There exists a linear function $\varphi$ so that
$\mathcal{Q}(\mathcal{M})=e^{\varphi}\operatorname{Span}\{\pone,\phi_1,\phi_2\}$ and so that
the map $\Phi:=(\phi_1,\phi_2)$ defines an immersion of $\mathbb{R}^2$ in $\mathbb{R}^2$. Any geodesic on $\mathcal{M}$
locally has the form $\sigma(t)=\Phi^{-1}(\psi_\sigma(t)u_\sigma+v_\sigma)$ for some smooth function $\psi_\sigma$ and for
suitably chosen vectors $u_\sigma$ and $v_\sigma$ in $\mathbb{R}^2$.
\end{theorem}

Theorem~\ref{T1.6} is only a local result; however, since we are working
in the real analytic setting, this does not affect our ansatz. This point arises in the analysis of Section~\ref{S6.1.6} for example.
Our study of the geodesic structure in Type~$\mathcal{A}$ geometries in Section~\ref{S7} will
be based on Theorem~\ref{T1.6} and upon a knowledge
of $\mathcal{Q}(\mathcal{M})$ which is an analytic invariant; it is not simply a straightforward exercise
in computer algebra. The geodesic equation is a linked pair of non-linear equations in 1-variable; Theorem~\ref{T1.6}
reduces consideration to finding a single function of 1-variable. This approach permits us to determine in Section~\ref{S7}
all the geodesics of the affine manifolds $\MM_i^j(\cdot)$ for $j=4$ and $j=6$; for $j=2$, we obtained
ODEs we could not solve although we did obtain sufficient information to establish whether or
not these geometries were geodesically complete. D'Ascanio et al. \cite{DGP17} determined
which non-flat Type~$\mathcal{A}$ geometries were geodesically complete using a very different approach.
In Section~\ref{S6}, we will establish the following result which extends their results by
taking into account the flat geometries; we believe it is a more straightforward treatment -- it also yields more
information.

\begin{theorem}\label{T1.7}
Let $\mathcal{M}=(\mathbb{R}^2,\nabla)$ be a Type~$\mathcal{A}$ surface.
Then $\mathcal{M}$ is geodesically complete if and only
$\mathcal{M}$ is linearly equivalent to $\MM_0^6$, to $\MM_4^6$, to $\MM_3^4(-\frac12)$,
or to $\MM_2^2(-1,a)$ for some $a$.
\end{theorem}

The affine Killing vector fields of a Type~$\mathcal{A}$ geometry are real analytic. From this it follows that if
 $\tilde\MM$ is an affine surface which is modeled on a Type~$\mathcal{A}$ geometry $\mathcal{M}=(\mathbb{R}^2,\nabla)$ (where $\nabla$
 has constant Christoffel symbols), then $\tilde\MM$ is real analytic.
We say that a Type~$\mathcal{A}$ structure $\mathcal{M}$ on $\mathbb{R}^2$ is essentially geodesically incomplete if there is
no surface $\tilde\MM$ which is modeled on $\MM$ and which is geodesically complete.
It will follow from the analysis of Section~\ref{S6} that
any non-flat Type~$\mathcal{A}$ structure on $\mathbb{R}^2$ which is geodesically incomplete but not essentially geodesically incomplete
is linearly equivalent either to $\MM_2^4(-\frac12)$
or to $\MM_5^4(0)$; up to linear equivalence, $\MM_2^4(-\frac12)$ and
$\MM_5^4(0)$ are the only non-flat Type~$\mathcal{A}$ structures on $\mathbb{R}^2$ which can
be geodesically completed. This is analogous to the situation when we considered the completion of affine Killing incomplete Type~$\mathcal{A}$
structures on $\mathbb{R}^2$.

\subsection{Type~$\mathcal{B}$ geometries}

$\nabla$ is a left invariant connection on the $ax+b$ group
$\mathbb{R}^+\times\mathbb{R}$ if and only if $\Gamma=(x^1)^{-1}\Gamma(a,b,c,d,e,f)$; we denote the
corresponding structure by $\mathcal{N}((x^1)^{-1}\Gamma(a,b,c,d,e,f))$.
\begin{definition}\label{D1.8}\rm We define the following Type~$\mathcal{B}$ affine structures $\mathcal{N}_i^j(\cdot)$
on $\mathbb{R}^+\times\mathbb{R}$; a direct computation then establishes $\mathcal{Q}$.
\par\quad$\mathcal{N}_0^6:=\Gamma(0,0,0,0,0,0)$, $\mathcal{Q}(\mathcal{N}_0^6)=\operatorname{Span}\{\pone,x^1,x^2\}$.
\par\quad$\mathcal{N}_1^6(\pm):=\mathcal{N}((x^1)^{-1}\Gamma(1,0,0,0,\pm1,0))$,
\par\qquad\qquad
$\mathcal{Q}(\mathcal{N}_1^6(\pm))=\operatorname{Span}\{\pone,x^2,(x^1)^2\pm(x^2)^2\}$.
\par\quad$\mathcal{N}_2^6(c):=\mathcal{N}((x^1)^{-1}\Gamma(c-1,0,0,c,0,0))$ for $c\ne0$,\par\qquad\qquad
$\mathcal{Q}(\mathcal{N}_2^6(c))=\operatorname{Span}\{\pone,(x^1)^c,(x^1)^cx^2\}$.
\par\quad$\mathcal{N}_3^6:=\mathcal{N}((x^1)^{-1}\Gamma(-2,1,0,-1,0,0))$,
$\mathcal{Q}(\mathcal{N}_3^6)=\operatorname{Span}\{\pone,\frac1{x^1},\frac{x^2}{x^1}+\log(x^1)\}$.
\par\quad$\mathcal{N}_4^6:=\mathcal{N}((x^1)^{-1}\Gamma(0,1,0,0,0,0))$,
$\mathcal{Q}(\mathcal{N}_4^6)=\operatorname{Span}\{\pone,x^1,x^2+x^1\log(x^1)\}$.
\par\quad$\mathcal{N}_5^6:=\mathcal{N}((x^1)^{-1}\Gamma(-1,0,0,0,0,0))$,
$\mathcal{Q}(\mathcal{N}_5^6)=\operatorname{Span}\{\pone,\log(x^1),x^2\}$.
\par\quad$\mathcal{N}_6^6(c):=\mathcal{N}((x^1)^{-1}\Gamma(c,0,0,0,0,0))$ for $c\ne0,-1$,\par\qquad\qquad
$\mathcal{Q}(\mathcal{N}_6^6(c))=\operatorname{Span}\{\pone,(x^1)^{1+c},x^2\}$.
\par\quad$\mathcal{N}_1^4(\kappa):=\mathcal{N}((x^1)^{-1}(2\kappa,1,0,\kappa,0,0))$ for $\kappa\notin\{0,-1\}$,
\par\qquad\qquad$\mathcal{Q}(\mathcal{N}_1^4(\kappa))=\operatorname{Span}\{(x^1)^\kappa, (x^1)^{\kappa+1}, (x^1)^\kappa(x^2+x^1\log x^1)\}$.
\par\quad$\mathcal{N}_2^4(\kappa,\theta):=\mathcal{N}((x^1)^{-1}\Gamma(
2\kappa+\theta-1,0,0,\kappa,0, 0))$ for $\kappa\notin\{0,-\theta\}$ and $\theta\ne0$,
\par\qquad\qquad$\mathcal{Q}(\mathcal{N}_2^4(\kappa,\theta))=\operatorname{Span}\{(x^1)^\kappa, (x^1)^\kappa x^2,  (x^1)^{\kappa+\theta} \}$.
\par\quad$\mathcal{N}_3^4(\kappa):=\mathcal{N}((x^1)^{-1}\Gamma(2\kappa-1,0,0,\kappa,0,0))$ for $\kappa\ne0$,
\par\qquad\qquad$\mathcal{Q}({\mathcal{N}_3^4(\kappa))}=\operatorname{Span}\{(x^1)^\kappa, (x^1)^\kappa x^2, (x^1)^\kappa \log x^1 \}$.
\par\quad$\mathcal{N}_1^3(\pm):=\mathcal{N}((x^1)^{-1}\Gamma(-\frac32,0,0,-\frac12,\mp\frac12,0))$,
$\mathcal{Q}(\mathcal{N}_1^3(\pm))=\{0\}$.
\par\quad$\mathcal{N}_2^3(c):=\mathcal{N}((x^1)^{-1}\Gamma(-\frac32,0,1,-\frac12,c,2))$,
$\mathcal{Q}(\mathcal{N}_2^3(c))=\{0\}$.
\par\quad$\mathcal{N}_3^3:=\mathcal{N}((x^1)^{-1}\Gamma(-1,0,0,-1,-1,0))$,
$\mathcal{Q}(\mathcal{N}_3^3)=\operatorname{Span}\{\frac1{x^1},\frac{x^2}{x^1},\frac{(x^2)^2-(x^1)^2}{x^1}\}$.
\par\qquad\qquad This is the affine structure of the
Lorentzian-hyperbolic plane given by\par\qquad\qquad  the metric $ds^2=\frac{(dx^1)^2-(dx^2)^2}{(x^1)^2}$.
\par\quad$\mathcal{N}_4^3:=\mathcal{N}((x^1)^{-1}\Gamma(-1,0,0,-1,1,0))$,
$\mathcal{Q}(\mathcal{N}_4^3)=\operatorname{Span}\{\frac1{x^1},\frac{x^2}{x^1},\frac{(x^2)^2+(x^1)^2}{x^1}\}$.
\par\qquad\qquad This is the affine structure of the hyperbolic plane given by the metric
\par\qquad\qquad $ds^2=\frac{(dx^1)^2+(dx^2)^2}{(x^1)^2}$.
\end{definition}

We refer to Brozos-V\'{a}zquez et al.~\cite{BGG18} for the proof of Assertions~(1--3) in Theorem~\ref{T1.9} below.
Assertion~(4) will be established in Section~\ref{S8} and is
the appropriate generalization of Theorem~\ref{T1.4}~(4) to this setting; unlike
the case of the Type~$\mathcal{A}$ geometries, there is no classification for the generic case
$\dim\{\mathfrak{K}(\mathcal{N})\}=2$ and this is why the determination of which of these geometries is
geodesically complete is unsettled.

\begin{theorem}\label{T1.9}
Let $\mathcal{N}$ be a Type~$\mathcal{B}$ structure on $\mathbb{R}^+\times\mathbb{R}$.
\begin{enumerate}
\item $\dim\{\mathfrak{K}\}\in\{2,3,4,6\}$.
\item $\dim\{\mathfrak{K}(\mathcal{N})\}=3$ if and only if $\mathcal{N}$ is linearly equivalent to $\mathcal{N}_i^3(\cdot)$ for some $i$.
\item The following assertions are equivalent.
\begin{enumerate}
\item $\dim\{\mathfrak{K}(\mathcal{N})\}=4$.
\item $\mathcal{N}$ is linearly equivalent to $\mathcal{N}_i^4(\cdot)$ for some $i$;
\item $\mathcal{N}$ is also Type~$\mathcal{A}$.
\end{enumerate}
\item The following assertions are equivalent.
\begin{enumerate}
\item $\dim\{\mathfrak{K}(\mathcal{N})\}=6$.
\item $\mathcal{N}$ is linearly equivalent to $\mathcal{N}_i^6(\cdot)$ for some $i$.
\item $\mathcal{N}$ is flat.\end{enumerate}\end{enumerate}
\end{theorem}
We will prove the following result in Section~\ref{S8}.

\begin{theorem}\label{T1.10}
Let $\mathcal{M}$ be a Type~$\mathcal{B}$ structure on $\mathbb{R}^+\times\mathbb{R}$.
\begin{enumerate}
\item If $\dim\{\mathfrak{K}(\mathcal{M})\}=2$, then $\mathcal{M}$ is affine Killing complete.
\item If $\dim\{\mathfrak{K}(\mathcal{M})\}=3$, then $\mathcal{M}$ is affine Killing complete if and only if
$\mathcal{M}$ is linearly equivalent to the hyperbolic plane.
\item If $\dim\{\mathfrak{K}(\mathcal{M})\}=4$, then $\mathcal{M}$ is affine Killing complete.
\item If $\dim\{\mathfrak{K}(\mathcal{M})\}=6$, then $\mathcal{M}$ is affine Killing complete if and only
if $\mathcal{M}$ is linearly equivalent to $\mathcal{N}_0^6$ or $\mathcal{N}_5^6$.\end{enumerate}
\end{theorem}

The question of geodesic completeness is more subtle and will not be dealt with here since, unlike the
Type~$\mathcal{A}$ geometries, we do not have a suitable parametrization of the Type~$\mathcal{B}$ surfaces
where $\dim\{\mathfrak{K}(\mathcal{N})\}=2$.

\section{The proof of Lemma~\ref{L1.1}}\label{S2}
Let $\mathcal{M}=(\mathbb{R}^2,\nabla)$ be a Type~$\mathcal{A}$ geometry.
An affine surface $\mathcal{M}$ is strongly projectively flat if and only if both
$\rho$ and $\nabla\rho$ are totally symmetric (see, for example, Eisenhart~\cite{E70} or Nomizu and Sasaki~\cite{NS}). A direct computation shows that
$\rho$ and $\nabla\rho$ are in fact totally symmetric
if $\mathcal{M}$ is Type~$\mathcal{A}$ and thus every Type~$\mathcal{A}$ surface is strongly projectively flat.
However, this argument does not show that the associated flat surface is again Type~$\mathcal{A}$ nor does it show that
the equivalence can be obtained using a linear function. We proceed as follows.
Let $\varphi(x^1,x^2)=a_1x^1+a_2x^2$ for $(a_1,a_2)\in\mathbb{R}^2$. Let $\tilde{\mathcal{M}}:={}^{-\varphi}\mathcal{M}$
have Ricci tensor $\tilde\rho$. We wish to choose $(a_1,a_2)$ so $\tilde\rho=0$. We suppose first that $\Gamma_{11}{}^2\ne0$. By
rescaling $x^2$, we may assume that $\Gamma_{11}{}^2=1$. We solve the equation $\tilde\rho_{11}=0$ for $a_2$ to obtain
$$
a_2:=a_1^2 - a_1 \Gamma_{11}{}^1 - \Gamma_{12}{}^1 + \Gamma_{11}{}^1 \Gamma_{12}{}^2 - (\Gamma_{12}{}^2)^2 + \Gamma_{22}{}^2\,.
$$
This yields
\begin{eqnarray*}
\tilde\rho_{12}&=&-\Gamma_{22}{}^1 + (a_1 - \Gamma_{12}{}^2) (a_1^2 - a_1 \Gamma_{11}{}^1 - 2 \Gamma_{12}{}^1
+ \Gamma_{11}{}^1 \Gamma_{12}{}^2 - (\Gamma_{12}{}^2)^2 +
    \Gamma_{22}{}^2)\\
\tilde\rho_{22}&=&(a_1 - \Gamma_{11}{}^1 + \Gamma_{12}{}^2)\tilde\rho_{12}\,.
\end{eqnarray*}
The crucial point is that $\tilde\rho_{12}$ divides $\tilde\rho_{22}$. Thus it suffices to choose $a_1$ so $\tilde\rho_{12}=0$.
Since $\tilde\rho_{12}$ is a monic polynomial of $a_1$,
we can find $a_1$ so $\tilde\rho_{12}=0$.  We now have $\tilde\rho=0$ so $\tilde{\mathcal{M}}$ is flat as desired.

We suppose next that $\Gamma_{11}{}^2=0$. If $\Gamma_{22}{}^1\ne0$, we can interchange the roles of $x^1$ and $x^2$
and repeat the argument given above. We may therefore assume that $\Gamma_{22}{}^1=0$ as well. We make a direct computation to see
that taking $a_1=\Gamma_{12}{}^2$
and $a_2=\Gamma_{12}{}^1$ yields $\tilde\rho=0$. Since
$\pone\in\mathcal{Q}(\tilde{\mathcal{M}})$, we conclude $e^\varphi\in\mathcal{Q}(\mathcal{M})=e^\varphi\mathcal{Q}(\tilde{\mathcal{M}})$.\qed

\section{Affine embeddings and immersions of Type~$\mathcal{A}$ structures}\label{S3}
We introduce an auxiliary affine surface $\tilde{\MM}_5^4(c):=(\mathbb{R}^2,\nabla)$
where the only (possibly) non-zero Christoffel symbols of $\nabla$ are
 $\Gamma_{22}{}^1=(1+c^2)x^1$ and $\Gamma_{22}{}^2=2c$; this is not a Type~$\mathcal{A}$ structure on $\mathbb{R}^2$. We have
$$
\mathcal{Q}(\tilde{\MM}_5^4(c))=\operatorname{Span}\{e^{cx^2}\cos(x^2),e^{cx^2}\sin(x^2),x^1\}\,.
$$
We will
show presently in Section~\ref{S3.3} that $\operatorname{Aff}(\tilde{\MM}_5^4(c))$ acts transitively on $\mathbb{R}^2$
and consequently this is a homogeneous geometry.
\begin{theorem}\label{T3.1}
\ \begin{enumerate}
\item $\PP_1^6(x^1,x^2):=(e^{x^1},x^2e^{x^1})$ is an affine embedding of $\MM_1^6$ in $\MM_0^6$.
\item $\PP_2^6(x^1,x^2):=(e^{x^2},e^{-x^1})$ is an affine embedding of $\MM_2^6$ in $\MM_0^6$.
\item $\PP_3^6(x^1,x^2):=(x^1,e^{x^2})$ is an affine embedding of $\MM_3^6$ in $\MM_0^6$.
\item $\PP_4^6(x^1,x^2):=(x^2,(x^2)^2+2x^1)$ is an affine isomorphism from $\MM_4^6$ to $\MM_0^6$.
\item $\PP_5^6(x^1,x^2)=(e^{x^1}\cos(x^2),e^{x^1}\sin(x^2))$ is an affine immersion of $\MM_5^6$ in
$\MM_0^6$.
\item $\PP_1^4(x^1,x^2):=(e^{-x^1},x^2)$ is an affine embedding of $\MM_1^4$ in $\MM_4^4(0)$.
\item $\PP_2^4(x^1,x^2):=(e^{-x^1},x^2)$ is an affine embedding of $\MM_2^4(c)$ in $\MM_3^4(c)$.
\item $\PP_4^4(c)(x^1,x^2):=(x^1+\frac12c(x^2)^2,x^2)$ is an affine isomorphism from $\MM_4^4(c)$ to $\MM_4^4(0)$.
\item $\PP_5^4(c)(x^1,x^2):=(e^{x^1},x^2)$ is an affine embedding of $\MM_5^4(c)$ in $\tilde{\MM}_5^4(c)$.
\end{enumerate}
\end{theorem}

\begin{proof} Because $\dim\{\mathcal{Q}(\cdot)\}=3$,
the geometries in question are all strongly projectively flat.
The affine maps in question intertwine the solution spaces $\mathcal{Q}(\cdot)$.
Thus Theorem~\ref{T3.1} follows from Theorem~\ref{T1.2}.
\end{proof}

\section{The proof of Theorem~\ref{T1.5}}\label{S4}
Let $\mathcal{M}$ be a Type~$\mathcal{A}$ surface model. We have $\dim\{\mathfrak{K}(\mathcal{M})\}\in\{2,4,6\}$.
If $\dim\{\mathfrak{K}(\mathcal{M})\}=2$,  then
$\mathfrak{K}(\mathcal{M})=\operatorname{Span}\{\partial_{x^1},\partial_{x^2}\}$. The flow lines
of the affine Killing vector fields are straight lines with a linear parametrization and are complete
so Theorem~\ref{T1.5} is immediate. If $\dim\{\mathfrak{K}(\mathcal{M})\}=6$, then $\mathcal{M}$
is flat. We apply Theorem~\ref{T3.1};
$\PP_i^6$ is a diffeomorphism if $i=4$ and thus $\MM_4^6$ is
affine Killing complete. $\PP_i^6$ is not surjective if  $i=1,2,3,5$ and thus $\MM_i^6$ is affine incomplete in
for these values of $i$. We complete the proof of Theorem~\ref{T1.5} by dealing with the case
$\dim\{\mathfrak{K}(\mathcal{M})\}=4$.
By Theorem~\ref{T1.4} and Theorem~\ref{T3.1} we may assume that $\mathcal{M}=\MM_3^4(c)$, that
$\mathcal{M}=\MM_4^4(0)$, or to replace $\mathcal{M}$ by $\tilde{\MM}_5^4(c)$. We examine these 3
cases seratim.
\subsection{Case 1. $\boldsymbol{\MM_3^4(c)}$} We have
$\mathcal{Q}(\MM_3^4(c))=\operatorname{Span}\{e^{cx^2},e^{(1+c)x^2},x^1e^{cx^2}\}$. This is
not a particularly convenient form of this surface to work with. We set $u^1:=x^1e^{cx^2}$ and $u^2:=x^2$
to express
$\mathcal{Q}(\MM_3^4(c))=\operatorname{Span}\{e^{cu^2},e^{(1+c)u^2},u^1\}$. We define
$$
T(a_1,b_1,c_1,c_2)(u^1,u^2)=(e^{a_1}u^1+b_1e^{cu^2}+c_1e^{(1+c)u^2},u^2+d_1)\,.
$$
Because $T(a_1,b_1,c_1,d_1)^*\mathcal{Q}(\MM_3^4(c))=\mathcal{Q}(\MM_3^4(c))$,
$T(a_1,b_1,c_1,d_1)$ defines a diffeomorphism of $\mathbb{R}^2$ preserving the affine structure.
We verify
\begin{eqnarray*}
&&
T(a_1,b_1,c_1,d_1)\circ T(a_2,b_2,c_2,d_2)\\
&&\quad=T(a_1+a_2,b_2e^{a_1}+b_1e^{cd_2},c_2e^{a_1}+c_1e^{(1+c)d_2},d_1+d_2){\,,}\\
&&T(a_1,b_1,c_1,d_1)^{-1}=T(-a_1,-b_1e^{-a_1-cd_1},-c_1e^{-a_1+(-1-c)d_1},-d_1)\,.
\end{eqnarray*}
This gives $\mathbb{R}^4$ the structure of a Lie group and constructs a 4-parameter family of affine Killing vector fields
which for dimensional reasons must be $\mathfrak{K}(\MM_3^4(c))$ and thereby shows $\MM_3^4(c)$
is affine Killing complete.

\subsection{Case 2. $\boldsymbol{\MM_4^4(0)}$} We have
$\mathcal{Q}(\MM_4^4(0))=\operatorname{Span}\{e^{x^2},x^2e^{x^2},x^1e^{x^2}\}$. We clear the previous notation and set
$$
T(a_1,b_1,c_1,d_1)(x^1,x^2):=(e^{a_1}x^1+b_1x^2+c_1,x^2+d_1)\,.
$$
Since $T(a_1,b_1,c_1,d_1)^*\mathcal{Q}(\MM_4^4(0))=\mathcal{Q}(\MM_4^4(0))$,
$T(a_1,b_1,c_1,d_1)$ is a diffeomorphism of $\mathbb{R}^2$ preserving the affine structure.
The group structure on $\mathbb{R}^4$ is given by
\begin{eqnarray*}
&&T(a_1,b_1,c_1,d_1)\circ T(a_2,b_2,c_2,d_2)\\
&&\quad=T(a_1+a_2,b_1+b_2e^{a_1},c_1+b_1d_2+c_2e^{a_1},d_1+d_2),\\
&&T(a_1,b_1,c_1,d_1)^{-1}=T(-a_1,-b_1e^{-a_1},e^{-a_1}(b_1d_1-c_1),-d_1)\,.
\end{eqnarray*}
It now follows $\MM_4^4(0)$ is affine Killing complete.

\subsection{Case 3. $\boldsymbol{\tilde{\MM}_5^4(c)}$}\label{S3.3}
We have
$\mathcal{Q}(\tilde{\MM}_5^4(c))=\operatorname{Span}\{e^{cx^2}\cos(x^2),e^{cx^2}\sin(x^2),x^1\}$.
We clear the previous notation and set
$$
T(a_1,b_1,c_1,d_1)(x^1,x^2):=(e^{a_1}x^1+b_1e^{cx^2}\cos(x^2)+c_1e^{cx^2}\sin(x^2),x^2+d_1)\,.
$$
Then $T(a_1,b_1,c_1,d_1)^*\mathcal{Q}(\tilde{\MM}_5^4(c))=\mathcal{Q}(\tilde{\MM}_5^4(c))$ so
$T(a_1,b_1,c_1,d_1)$ is a diffeomorphism of $\mathbb{R}^2$ preserving the affine structure. The group
structure is given by
\begin{eqnarray*}
&&T(a_1,b_1,c_1,d_1)\circ T(a_2,b_2,c_2,d_2)\\
&&\quad= T(a_1+a_2,e^{a_1} b_2+b_1 e^{c d_2} \cos (d_2)+c_1 e^{c d_2} \sin (d_2),\\
&&\qquad e^{a_1} c_2-b_1 e^{c d_2} \sin (d_2)+c_1 e^{c d_2} \cos (d_2),d_1+d_2),\\
&&T(a_1,b_1,c_1,d_1)^{-1}=T(-a_1,-e^{-a_1-c d_1} (b_1 \cos (d_1)-c_1 \sin (d_1)),\\
&&\qquad\qquad\qquad\qquad\qquad -e^{-a_1-c d_1} (b_1 \sin (d_1)+c_1 \cos (d_1)),-d_1)\,.
\end{eqnarray*}
It now follows $\tilde{\MM}_5^4(c)$ is affine Killing complete. It is immediate that $\operatorname{Aff}(\tilde{\MM}_5^4(c))$ acts
transitively on $\mathbb{R}^2$ so this is a homogeneous geometry.

\section{The proof of Theorem~\ref{T1.6}}\label{S5}
Let $\mathcal{M}=(\mathbb{R}^2,\nabla)$ be a Type~$\mathcal{A}$ structure on $\mathbb{R}^2$.
By Lemma~\ref{L1.1}, there exists a linear function $\varphi$ with $e^\varphi\in\mathcal{Q}(\mathcal{M})$ and so
$\tilde{\mathcal{M}}:={}^{-\varphi}\mathcal{M}$ is flat.

Since $e^{\varphi}\in\mathcal{Q}(\mathcal{M})$ and $\dim\{\mathcal{Q}(\mathcal{M})\}=3$, we have
$\mathcal{Q}(\mathcal{M})=e^{\varphi}\operatorname{Span}\{\pone,\phi_1,\phi_2\}$.
Set $\Xi_P(\phi):=\{\phi,\partial_{x^1}\phi,\partial_{x^2}\phi\}(P)$ for $P\in\mathbb{R}^2$.
By Theorem~\ref{T1.2}, $\Xi_P$ is an injective map from $\mathcal{Q}(\mathcal{M})$ to $\mathbb{R}^3$.
Since $\dim\{\mathcal{Q}(\mathcal{M})\}=3$, $\Xi_P$ is bijective.
It now follows that $d\phi_1(P)$ and $d\phi_2(P)$ are linearly independent so $\Phi:=(\phi_1,\phi_2)$ is an immersion.

By Theorem~\ref{T1.2}, $\mathcal{Q}(\tilde{\mathcal{M}})=
\operatorname{Span}\{\pone,\phi_1,\phi_2\}=\Phi^*\operatorname{Span}\{\pone,x^1,x^2\}=\Phi^*\mathcal{Q}(\MM_0^6)$. Consequently,
$\tilde{\mathcal{M}}=\Phi^*\MM_0^6$ by Theorem~\ref{T1.2}.
The affine geodesics in $\MM_0^6$ are straight lines and can be written in the form $t u+v$ for $u$ and $v$ in $\mathbb{R}^2$.
Thus the affine geodesics in $\tilde{\mathcal{M}}$ locally take the form $\Phi^{-1}(tu+v)$. Since $\mathcal{M}$ and $\tilde{\mathcal{M}}$
are strongly projectively equivalent, the unparameterized geodesics of $\mathcal{M}$ and $\tilde{\mathcal{M}}$ agree. The desired result now
follows.\qed

\section{The proof of Theorem~\ref{T1.7}}\label{S6}

We divide the proof of Theorem~\ref{T1.7} into 3 cases depending on $\operatorname{Rank}\{\rho\}$ or, equivalently, on
$\dim\{\mathfrak{K}\}$; each is then divided further
depending on the particular family involved.
We use the ansatz of Theorem~\ref{T1.6}.  Let $\sigma_{a,b}(t)$ be the
affine geodesic with $\sigma_{a,b}(0)=0$ and $\dot\sigma_{a,b}(0)=(a,b)$.
\subsection{Case 1. The flat geometries $\boldsymbol{\MM_i^6}$}
These geometries all are locally affine equivalent to the affine plane $(\mathbb{R}^2,\Gamma_0^6)$;
this geodesically complete affine surface provides a local model for each of these geometries.
$\PP_i^6$ for $i=1,2,3$ embeds $\MM_i^6$ in $\MM_0^6$, $\PP_4^6$ provides
a diffeomorphism between $\MM_4^6$ and $\MM_0^6$, and $\PP_5^6$ immerses
$\MM_5^6$ in $\MM_0^6$. Thus $\MM_i^6$ is geodesically incomplete for
$i=1,2,3,5$ and $\MM_4^6$ is geodesically complete.
\subsubsection{$\boldsymbol{\MM_0^6}$}
$\sigma_{a,b}(t)=(at,bt)$. $\MM_0^6$ is geodesically complete and
defines the flat affine plane $\mathbb{A}^2$.

\subsubsection{$\boldsymbol{\MM_1^6}$ }
$\sigma_{a,b}(t)=(\log(1+at),\frac{bt}{1+at})$. $\MM_1^6$ is geodesically
incomplete; $\sigma_{a,b}(t)$ is defined for all $t\in\mathbb{R}$ if and only if $a=0$.

\subsubsection{$\boldsymbol{\MM_2^6}$}
$\sigma_{a,b}(t)=(-\log(1-at),\log(1+bt))$.
$\MM_2^6$ is geodesically incomplete; no non-trivial geodesic is defined for all $t\in\mathbb{R}$.

\subsubsection{$\boldsymbol{\MM_3^6}$}
$\sigma_{a,b}(t)=(at,\log(1+bt))$. $\sigma_{a,b}(t)$ is defined for all $t\in\mathbb{R}$ if and only if $b=0$.

\subsubsection{$\boldsymbol{\MM_4^6}$} $\sigma_{a,b}(t)=(at-\frac12b^2t^2,bt)$. $\MM_4^6$ is geodesically complete.

\subsubsection{$\boldsymbol{\MM_5^6}$}\label{S6.1.6}
$\sigma_{a,b}(t)=\left(\frac12\log((1+at)^2+b^2t^2),\arctan\left(\frac{tb}{1+at}\right)\right)$.
$\sigma_{a,b}(t)$ extends to be defined for all $t\in\mathbb{R}$
if and only if $b\ne0$.

\subsection{Case 2: The geometries $\boldsymbol{\MM_i^4(\cdot)}$}
For these geometries, the Ricci tensor is a non-zero constant multiple $\lambda$
of $dx^2\otimes dx^2$. Suppose there exists a geodesically complete affine
surface $\tilde{\mathcal{M}}$ which is modeled on $\MM_i^4(\cdot)$.
Let $\sigma$ be a small piece of a geodesic in $\MM_i^4(\cdot)$
defined by $\sigma(t)=(x^1(t),x^2(t))$ which can be copied into $\tilde{\mathcal{M}}$. Then $\rho(\dot\sigma,\dot\sigma)(t)=\lambda(\dot x^2(t))^2$
extends to a real analytic function on $\tilde{\mathcal{M}}$
which is defined for all $t$. If we can exhibit a geodesic where $\dot x^2(t)$ is not bounded,
it then follows that $\MM_i^4(\cdot)$ is essentially geodesically incomplete.

\subsubsection{$\boldsymbol{\MM_1^4}$} This geometry is essentially geodesically incomplete because
$$\sigma_{a,b}(t)=\left\{\begin{array}{ll}
(-\log (1-\frac{a \log (2 b t+1)}{2 b}),\frac{1}{2} \log (2 b t+1))&\text{ if }b\ne0\\
(-\log(1-at),0)&\text{ if }b=0\end{array}\right\}\,.$$

\subsubsection{$\boldsymbol{\MM_2^4(c)}$}
If $c\ne-\frac12$, then $\sigma_{a,b}(t)$ is given by:
$$
\left\{\begin{array}{ll}
(\log (\frac{b}{a+b})-\log (1-\frac{a (2 b c t+b t+1)^{\frac{1}{2 c+1}}}{a+b}),
\frac{\log (2 b c t+b t+1)}{2 c+1})&\text{if }b\ne0,b\ne-a\\
\frac{(-\log(1+b(t+2ct)),\log(1+bt+2bct))}{1+2c}&\text{if }b\ne0,b=-a\\
(-\log(1-at),0)&\text{if }b=0
\end{array}\right\}.$$
This geometry is essentially geodesically incomplete. If $c=-\frac12$, then
$$
\sigma_{a,b}(t)=\left\{\begin{array}{ll}
(-\log(1-at),0)&\text{if }b=0\\
(-\log(a(e^{bt}-1)-b)+\log(-b),bt)&\text{if }b<0\\
(-\log(-a(e^{bt}-1)+b)+\log(b),bt)&\text{if }b>0
\end{array}\right\}\,.
$$
This geometry is geodesically incomplete. By Theorem~\ref{T3.1}, there is an affine
embedding of $\MM_2^4(-\frac12)$ in $\MM_3^4(-\frac12)$. Since
we shall show presently that $\Gamma_3^1(-\frac12)$ is geodesically
complete, the geometry $\MM_2^4(-\frac12)$ can be geodesically completed.

\subsubsection{$\boldsymbol{\MM_3^4(c)}$}
If $c\ne-\frac12$, let $\kappa:={1+2c}$. This geometry is essentially geodesically incomplete since
$$
\sigma_{a,b}(t)=\left\{\begin{array}{ll}
{\left(\frac ab((1+b t\kappa)^{{1}/{\kappa}}-1),\kappa^{-1}\log(1+b t\kappa)\right)}&\text{ if }b\ne0\\
(at,0)&\text{ if }b=0
\end{array}\right\}\,.
$$
 If $c=-\frac12$, then this geometry is geodesically complete since
$$
\sigma_{a,b}(t)=\left\{\begin{array}{ll}
(\frac ab(e^{bt}-1),bt)&\text{ if }b\ne0\\
(at,0)&\text{ if }b=0
\end{array}\right\}\,.
$$

\subsubsection{$\boldsymbol{\MM_4^4(c)}$} This geometry is essentially geodesically incomplete
since
$$\textstyle\sigma_{a,b}(t)=\left\{\begin{array}{ll}
(at,0)&\text{ if }b=0\\
(-\frac1{8b}\log(1+2bt)(-4a+bc\log(1+2bt)),\frac12\log(1+2bt))&\text{ if }b\ne0\end{array}\right\}\,.
$$

\subsubsection{$\boldsymbol{\MM_5^4(c)}$} This has an affine embedding in $\tilde{\MM}_5^4(c)$
which is not surjective and hence $\MM_5^4(c)$ is geodesically incomplete.

\subsubsection{$\boldsymbol{\tilde{\MM}_5^4(c)}$} We have $\rho=(1+c^2)dx^2\otimes dx^2$.
If $c=0$, then
$$
\sigma_{a,b}(t)=\left\{\begin{array}{ll}(at,0)&\text{ if }b=0\\
(\frac ab\sin(bt),bt)&\text{ if }b\ne0\end{array}\right\}\,.
$$
This geometry is geodesically complete. If $c\ne0$, then this geometry, and hence the geometry $\MM_5^4(c)$,
is esentially geodesically incomplete because
$$
\sigma_{a,b}(t)=\left\{\begin{array}{ll}
(at,0)&\text{ if }b=0\\
\left(\frac ab(1+2bct)^{1/2}\sin(\frac{\log(1+2bct)}{2c}),\frac{\log(1+2bct)}{2c}\right)&\text{ if }b\ne0\end{array}\right\}\,.
$$
  If $b\ne0$, then $\dot x^2$ does not remain bounded for all $t$. Thus all geodesics but one can not be completed. Since $\MM_5^4(c)$
embeds as an open subset of $\tilde{\MM}_5^4(c)$, this shows $\MM_5^4(c)$ also is essentially geodesically incomplete for $c\ne0$.

\subsection{Case 3. The geometries $\boldsymbol{\MM_i^2(\cdot)}$}
Suppose that $\tilde{\mathcal{M}}$ is a simply connected complete affine surface which is locally modeled on
$\MM_i^2(\cdot)$. Since
$\dim\{\mathfrak{K}(\tilde{\mathcal{M}})\}=2$, $\partial_{x^1}$ and
$\partial_{x^2}$ extend as Killing vector fields to all of $\mathcal{M}$. This shows that if $\gamma$ is a geodesic in
$\MM_i^2(\cdot)$, then $\rho(\dot\gamma,\partial_{x^i})$ is a bounded function on $\gamma$. We use this criteria
in what follows. In all cases, attempting to find the most general geodesic resulted in an ODE that we could
not solve explicitly.

\subsubsection{$\boldsymbol{\MM_1^2(a_1,a_2)}$}
We obtain 3 possible geodesics $\sigma_i(t)=\log(t)\vec\alpha_i$ where
$$
\vec\alpha_1=\frac{(1,1)}{1+a_1+a_2},\quad
\vec\alpha_2=\frac{(1-a_2,a_1)}{1+a_1-a_2},\quad
\vec\alpha_3=\frac{(a_2,1-a_1)}{1-a_1+a_2}.
$$
The first geodesic is defined for $a_1+a_2+1\ne0$, the second for $a_1-a_2+1\ne0$, and the third for $-a_1+a_2+1\ne0$. At least
two geodesics are defined for any given geometry. We have $\dot\sigma=\frac1t(c,d)$ for $(0,0)\ne(c,d)\in\mathbb{R}^2$.
Thus this geometry is essentially geodesically incomplete.
\subsubsection{$\boldsymbol{\MM_2^2(a_1,a_2)}$} Suppose $a_1\ne-1$.
We have a geodesic $\sigma(t)=\log(t)(\frac1{1+a_1},0)$. We conclude the geometry is essentially geodesically incomplete.
Suppose $a_1=-1$. We adapt an argument of Bromberg and Medina~\cite{B05}.
The geodesic equations become $\dot u=v(2a\dot u-\textstyle\frac12(1+a^2)v)$ and $\dot v=v(2u)$
or in matrix form:
$$
A\left(\begin{array}{c}u\\v\end{array}\right)
=v\left(\begin{array}{c}\dot u\\ \dot v\end{array}\right)\text{ for }
A:=\left(\begin{array}{cc}-2a&\frac12(1+a^2)\\-2&0\end{array}\right)\,.
$$
If $v(t_0)=0$ for any point in the parameter range, then $u(t)=u(t_0)$ and $v(t)=0$ solve this ODE.
Thus we may suppose without loss of generality $v$ does not change sign.
Introduce a new parameter $\tau$ so $\partial_\tau t=v(t)$ and let $U(\tau)=u(t(\tau))$ and $V(\tau)=v(t(\tau))$. We have
\begin{equation}\label{E6.a}
\partial_\tau\left(\begin{array}{c}U\\V\end{array}\right)=A\left(\begin{array}{c}U\\V\end{array}\right)\,.
\end{equation}
The eigenvalues of $A$ are $-a\pm\sqrt{-1}$. We solve Equation~(\ref{E6.a}) to see
$$
\left(\begin{array}{cc}U\\V\end{array}\right)=e^{-\tau a}\left\{\cos(\tau)\left(\begin{array}{cc}c_1\\c_2\end{array}\right)
+\sin(\tau)\left(\begin{array}{cc}-ac_1+\frac12(1+a^2)c_2\\-2c_1+ac_2\end{array}\right)\right\}\,.
$$
Thus $V=e^{-\tau a}(c_2\cos(\tau)+(-2c_1+ac_2)\sin(\tau))$. Since $V$ never vanishes, $\tau$ is restricted to
a parameter range of length at most $\pi$. It now follows that the original geodesic is for all $t\in\mathbb{R}$.

\subsubsection{$\boldsymbol{\MM_3^2(c)}$ \bf and $\boldsymbol{\MM_4^2(\pm)}$} We have
$\sigma_{a,0}=(\frac12\log(1+2at),0)$
so these geometries are essentially geodesically incomplete.

\section{The classification of flat Type~$\mathcal{B}$ geometries}\label{S7}
This section is devoted to the proof of Theorem~\ref{T1.6}~(4); by Theorem~\ref{T1.2}, it suffices to
classify the relevant solution spaces of the quasi-Einstein equation. Let $\mathcal{Q}=\mathcal{Q}(\mathcal{N})$
where $\mathcal{N}$ is a flat Type~$\mathcal{B}$ structure on $\mathbb{R}^+\times\mathbb{R}$. We work
modulo the action of the shear group $(x^1,x^2)\rightarrow(x^1,ax^1+bx^2)$ for $b\ne0$. Let
$\Phi=(\phi^1,\phi^2)$ be a local affine map from $\mathcal{N}$ to $\mathcal{M}_0^6$. We then have
$$
\mathcal{Q}=\Phi^*\operatorname{Span}\{\pone,x^1,x^2\}=\operatorname{Span}\{\pone,\phi^1,\phi^2\}\,.
$$
Since $\Phi$ is a local diffeomorphism, $\partial_{x^1}\mathcal{Q}\ne\{0\}$ and $\partial_{x^2}\mathcal{Q}\ne\{0\}$.
This rules out certain possibilities.

The vector fields $X:=x^1\partial_{x^1}+x^2\partial_{x^2}$ and $Y:=\partial_{x^2}$ are
Killing vector fields and therefore preserve $\mathcal{Q}$; the action of the Lie algebra $\operatorname{Span}\{X,Y\}$ on
$\mathcal{Q}$ is crucial. We complexify and
set $\mathcal{Q}_{\mathbb{C}}:=\mathcal{Q}\otimes_{\mathbb{R}}\mathbb{C}$;
elements of $\mathcal{Q}$ may be obtained by taking the
real and imaginary parts of complex solutions.
Decompose $\mathcal{Q}_{\mathbb{C}}=\oplus_\lambda\mathcal{Q}_\lambda$ as the direct sum of the generalized
eigenspaces of $X$ where
$$
\mathcal{Q}_\lambda:=\{f\in\mathcal{Q}_{\mathbb{C}}:(X-\lambda)^3f=0\}\,.
$$
The commutation relation $[X,\partial_{x^2}]=-\partial_{x^2}$ implies that
$$
\partial_{x^2}\mathcal{Q}_\lambda\subset\mathcal{Q}_{\lambda-1}\,.
$$

Choose $\lambda$ and $f\in\mathcal{Q}_\lambda$ so $\partial_{x^2}f\ne0$.
This implies $\mathcal{Q}_{\lambda-1}\ne0$. Thus, for dimensional reasons, $\dim\{\mathcal{Q}_\mu\}\le2$ for all $\mu$ and consequently
$$
\mathcal{Q}_\mu=\{f\in\mathcal{Q}_{\mathbb{C}}:(X-\mu)^2f=0\}\,.
$$
Since $\dim\{\mathcal{Q}\}=3$, $\{\mathcal{Q}_\lambda,\mathcal{Q}_{\lambda-1},\mathcal{Q}_{\lambda-2},\mathcal{Q}_{\lambda-3}\}$
can not all be non-trivial and thus, in particular, $(\partial_{x^2})^3f=0$ for any $f\in\mathcal{Q}_\lambda$. This implies any element of $\mathcal{Q}$
is a polynomial of degree at most 2 in $x^2$ with coefficients which are smooth functions of $x^1$.
If $(X-\lambda)f=0$,
then $f$ is a sum of elements of the form $(x^1)^{\lambda-k}(x^2)^k$ for $k\le 2$. If $(X-\lambda)^2f=0$, then $f$ is a sum of
elements of the form $(x^1)^{\lambda-k}(x^2)^k$ and $(x^1)^{\lambda-k}(x^2)^k\log(x^1)$ for $k\le2$. Since
$\dim\{\mathcal{Q}_\lambda\}\le2$, this is
the most complicated Jordan normal form possible.
In principle, the parameter $\lambda$ could be complex. It will follow from our subsequent analysis that this is not the case.
We adopt the
notation of Definition~\ref{D1.8}.

\subsection*{Case 1}
Suppose first that there exists $f\in\mathcal{Q}$ which has degree at least 2 in $x^2$. Let $f\in\mathcal{Q}_\lambda$
satisfy $\partial_{x^2}^2f\ne0$. Then $\{f,\partial_{x^2}f,\partial_{x^2}^2f\}$ is a basis for $\mathcal{Q}$. This implies $\partial_{x^2}^2f=c\pone$
so $\lambda=2$. Since $f\in\mathcal{Q}_2$, $\partial_{x^2} f\in\mathcal{Q}_1$, and $\pone\in\mathcal{Q}_0$,
$\dim\{\mathcal{Q}_\mu\}\le1$ for all $\mu$ and there are no log terms. Thus
$f=(x^2)^2+ax^1x^2+b(x^1)^2$. We may replace $x^2$ by $\tilde x^2=x^2+\frac12ax^1$ to ensure $a=0$. Since $\mathcal{Q}=\operatorname{Span}\{f,2x^2,\pone\}$ and since
$\partial_{x^1}\{\mathcal{Q}\}\ne0$, $b\ne0$. Rescale $x^2$ and renormalize $f$ to assume that $f=(x^2)^2\pm(x^1)^2$ and obtain
$\mathcal{N}_1^6(\pm)$.

\medbreak We assume henceforth that every element of $\mathcal{Q}$ is at most linear in $x^2$.
Since $\partial_{x^2}\{\mathcal{Q}\}\ne\{0\}$, we can choose $\lambda$ so that
$f=a_0(x^1)x^2+a_1(x^1)\in\mathcal{Q}_\lambda$ for $a_0(x^1)\ne0$. This gives rise to the following possibilities.

\subsection*{Case 2} Suppose $\lambda\notin\{0,1\}$. Then $\mathcal{Q}_\lambda$,
$\mathcal{Q}_{\lambda-1}$, and $\mathcal{Q}_0$ are non-trivial and distinct; hence each is 1-dimensional and
$\mathcal{Q}=\mathcal{Q}_\lambda\oplus\mathcal{Q}_{\lambda-1}\oplus\mathcal{Q}_0$. If $\lambda$ is complex,
then $\mathcal{Q}_{\bar\lambda}$ is non-trivial and is not contained in $\mathcal{Q}_\lambda\oplus\mathcal{Q}_{\lambda-1}\oplus\mathcal{Q}_0$
which is impossible. Thus $\lambda$ is real,
as noted above. Since $\dim\{\mathcal{Q}_\lambda\}=1$, there are no $\log(x^1)$ terms and $f=(x^1)^{\lambda-1}x^2+(x^1)^{\lambda}c$.
Replacing $x^2$ by $x^2-cx^1$ then permits us to assume $f=(x^1)^{\lambda-1}x^2$ so
$\mathcal{Q}=\operatorname{Span}\{\pone,(x^1)^{\lambda-1},(x^1)^{\lambda-1}x^2\}$ for $\lambda\ne0,1$.
This is $\mathcal{N}_2^6(c)$ for $c=\lambda-1\notin\{-1,0\}$.
We will deal with $\mathcal{N}_2^6(-1)$ subsequently.

\subsection*{Case 3} Suppose $\lambda=0$ so that $f=a_0(x^1)x^2+a_1(x^1)\in\mathcal{Q}_0$.
We then have $a_0(x^1)=\partial_{x^2}f\in\mathcal{Q}_{-1}$.
We also have $\pone\in\mathcal{Q}_0$.
Thus $\mathcal{Q}_{-1}$ is 1-dimensional so, after rescaling, we may take $a_0(x^1)=(x^1)^{-1}$ and
consequently $f=\frac{x^2}{x^1}+\varepsilon\log(x^1)$.
If $\varepsilon=0$, we obtain $\mathcal{N}_2^6(-1)$. If $\varepsilon\ne0$, we can rescale to obtain $\mathcal{N}_3^6$.

\subsection*{Case 4} Suppose $\lambda=1$ so $f=a_0(x^1)x^2+a_1(x^1)\in\mathcal{Q}_1$. Express
$$
f=x^2+x^2\alpha\log(x^1)+\beta x^1+\gamma x^1\log(x^1)\,.
$$
If $\alpha\ne0$, then $X$ has non-trivial Jordan normal form on $\mathcal{Q}_1$ so $\dim\{\mathcal{Q}_1\}\ge2$.
Furthermore $\partial_{x^2}f=\alpha\log(x^1)\in\mathcal{Q}_0$; since $\pone\in\mathcal{Q}_0$,
$\dim\{\mathcal{Q}_0\}\ge2$.  This is false.
Thus $\alpha=0$. By replacing $x^2$ by $x^2+\beta x^1$, we may assume
$\beta=0$ and obtain
$f=x^2+\gamma x^1\log(x^1)$. If $\gamma\ne0$, then applying $(X-1)$ we see $x^1\in\mathcal{Q}_1$; this gives, after rescaling,
$\mathcal{N}_4^6$. Thus we may assume
$\gamma=0$ so $x^2\in\mathcal{Q}_1$. If $\dim\{\mathcal{Q}_1\}=2$, we obtain $\mathcal{N}_0^6$.
If $\dim\{\mathcal{Q}_0\}=2$, then $\log(x^1)\in\mathcal{Q}_0$ and we obtain $\mathcal{N}_5^6$.
Otherwise, we obtain $\mathcal{N}_6^6(c)$ for $c\ne0,-1$. This completes the classification of the
flat Type~$\mathcal{B}$ structures.\qed

\section{Affine embeddings and immersions of Type~$\mathcal{B}$ structures}
Define
$$
\begin{array}{lll}
\Psi_0^6(x^1,x^2)=(x^1,x^2),&\Psi_1^6(\pm1)(x^1,x^2)=(x^2,(x^1)^2\pm(x^2)^2),\\
\Psi_2^6(c)(x^1,x^2)=((x^1)^c,(x^1)^cx^2),&\Psi_3^6(x^1,x^2)=(\frac1{x^1},\frac{x^2}{x^1}+\log(x^1)),\\
\Psi_4^6(x^1,x^2)=(x^1,x^2+x^1\log(x^1)),&\Psi_5^6(x^1,x^2)=(\log(x^1),x^2),\\
\Psi_6^6(c)(x^1,x^2)=((x^1)^{1+c},x^2),&\Psi_1^4(x^1,x^2)=(x^2+x^1\log(x^1),\log(x^1)),\\
\Psi_2^4(\kappa,\theta)(x^1,x^2)=(x^2,\theta\log(x^1)),&\Psi_3^4(\kappa)(x^1,x^2)=(x^2,\kappa\log(x^1)).
\end{array}$$

\begin{theorem}\label{T8.1}
\ \begin{enumerate}
\item $\Psi_i^6(\cdot)$ is an affine embedding of $\mathcal{N}_i^6({\cdot})$ in $\MM_0^6$ for any $i$.
\item $\Psi_1^4$ is an affine isomorphism from $\mathcal{N}_1^4(\kappa)$ to $\MM_3^4(\kappa)$.
\item $\Psi_2^4(\kappa,\theta)$ is an affine isomorphism from $\mathcal{N}_2^4(\kappa,\theta)$ to
$\MM_3^4(\frac\kappa\theta)$.
\item $\Psi_3^4(\kappa)$ is an affine isomorphism from $\mathcal{N}_3^4(\kappa)$ to $\MM_4^4(0)$.
\end{enumerate}
\end{theorem}
\begin{proof}
These geometries are all
strongly projectively flat and the diffeomorphisms in question intertwine the solution spaces $\mathcal{Q}(\cdot)$.
Thus Theorem~\ref{T8.1} follows from Theorem~\ref{T1.2}.
\end{proof}

\section{The proof of Theorem~\ref{T1.10}}\label{S8}

This section is devoted to the proof of Theorem~\ref{T1.10}. We apply Theorem~\ref{T8.1}.
Let $\mathcal{N}$ be a Type~$\mathcal{B}$ structure on $\mathbb{R}^+\times\mathbb{R}$. We distinguish cases.
\subsection*{Case 1} If $\dim\{\mathfrak{K}\}=6$, then
$\mathcal{N}$ is linearly equivalent to $\mathcal{N}_i^6$. The map $\Psi_i^6$ is an affine embedding of $\mathcal{N}_i^6$ in
$\mathbb{R}^2$ with the flat structure. If $i\ne5$, the embedding is not surjective and $\mathcal{N}_i^6$ is
affine Killing incomplete;
if $i=5$, then $\Psi_i^6$ is an isomorphism so $\mathcal{N}_5^6$ is affine complete.

\subsection*{Case 2} If $\dim\{\mathfrak{K}(\mathcal{N})\}=4$, then $\mathcal{N}$ is linearly equivalent to
$\mathcal{N}_i^4(\cdot)$. The maps $\Psi_i^4$ are affine
isomorphisms of $\mathcal{N}_i^4(\cdot)$ with $\MM_3^4(\cdot)$ or $\MM_4^4(0)$; these are affine
Killing complete by Theorem~\ref{T1.5}.
\subsection*{Case 3} $\dim\{\mathfrak{K}(\mathcal{N})\}=3$. There exists $\sigma\in\{0,\pm1\}$ so that
$$
\mathfrak{K}(\mathcal{M})=\operatorname{Span}\{X:=2x^1x^2\partial_{x^1}+((x^2)^2+\sigma(x^1)^2)\partial_2,
x^1\partial_{x^1}+x^2\partial_{x^2},\partial_{x^2}\}\,.s
$$
\subsection*{Case 3a} $\mathcal{N}_1^3(\pm)$ or $\mathcal{N}_2^3(c)$. We have
$\sigma=0$. The curve $\xi(t)=(x^1(t),x^2(t))$ is a flow curve for $X$ means that
$\dot x^1(t)=2x^1(t)x^2(t)$ and $\dot x^2(t)=x^2(t)^2$. We take $\xi(t)=(t^{-2},-t^{-1})$ to solve these equations
and to see these structures are affine Killing incomplete.

\subsection*{Case 3b} $\mathcal{N}_3^3$. We have $\sigma=1$. The curve $\xi(t)=(x^1(t),x^2(t))$
is a flow curve for $X$ means that
$\dot x^1=2x^1x^2$ and $\dot x^2=(x^2)^2+(x^1)^2$. We solve these equations by taking
$x^1(t)=-\frac12t^{-1}$ and $x^2(t)=-\frac12t^{-1}$.
Consequently, this structure is affine Killing incomplete.
This structure is the Lorentzian-hyperbolic plane; it
isometrically embeds in the pseudo-sphere which is affine complete.
We refer to \cite{APG18} for a further discussion of these two geometries and to \cite{G17} for a discussion of the pseudo-group
of isometries.
\subsection*{Case 3c} $\mathcal{N}_4^3$.  We have $\sigma=-1$. This is the hyperbolic plane and is
affine Killing complete.
\subsection*{Case 4} $\dim\{\mathfrak{K}(\mathcal{N})\}=2$.
${\mathfrak{K}}(\mathcal{N})=\operatorname{Span}\{x^1\partial_{x^1}+x^2\partial_{x^2},\partial_{x^2}\}$
and $\mathcal{N}$ is affine Killing complete.

\subsection*{Acknowledgments} We are grateful for useful comments by D. D'Ascanio and P. Pisani concerning these matters.

\end{document}